\documentclass[11pt,a4paper]{article}

\usepackage[utf8]{inputenc}
\usepackage{amsmath}
\usepackage{amssymb}
\usepackage{amsthm}
\usepackage{hyperref}

\usepackage{orcidlink}
\usepackage{authblk,amsfonts, stmaryrd, url,float}

 \newtheorem{theorem}{Theorem}[section]
 
  \newtheorem{proposition}[theorem]{Proposition}
 
 \newtheorem{remark}[theorem]{Remark}

\def\Star{\ensuremath{{\bf(*)}}}
\def\DCP{\ensuremath{{\bf(DCP_{A})}}}

\def\CEP{\ensuremath{{\bf(CEP)}}}
\def\iDEP{\ensuremath{{\bf(DEP1)}}}
\def\iiDEP{\ensuremath{{\bf(DEP2)}}}

\def\R{\mathbb{R}}
\def\F{\mathbb{F}}
\def\P{\mathbb{P}}
\def\N{\mathbb{N}}

\def\MCP{\ensuremath{{\bf(MCP)}}}
\def\FCP{\ensuremath{{\bf(FCP)}}}
\def\DCP{\ensuremath{{\bf(DCP)}}}

\def\DFP{\ensuremath{{\bf(DFP)}}}
\def\CPP{\ensuremath{{\bf(CPP)}}}

\def\LIP{\ensuremath{{\bf(LIP)}}}
\def\LMP{\ensuremath{{\bf(LMP)}}}

\def\lbrak{\left[}
\def\rbrak{\right]}
\def\lpar{\left(}
\def\rpar{\right)}
\def\abClosed{\lbrak a,b\rbrak}

\def\stepLatt{\mathcal{L}}
\def\elemLatt{\stepLatt(\abClosed)}
\def\elemI{I_{\abClosed}}
\def\istepLatt{\stepLatt_u}

\def\into{\longrightarrow}
\def\sub{\subseteq}
\def\setdiff{\backslash}

\def\nseqlim{{\displaystyle\lim_{n\rightarrow\infty}}}
\def\nseqliminf{{\displaystyle\liminf_{n\rightarrow\infty}}}

\def\abAlg{\mathcal{A}}

\def\Cara{Carath\'{e}odory\  }

\usepackage{tikz} 
\usetikzlibrary{arrows.meta, positioning}

\title{Extension Theorems in Measure and Integration Theory as Completeness Axioms}
\author{Rafael Reno S. Cantuba\thanks{Associate Professor, Department of Mathematics and Statistics, De La Salle University (DLSU), Taft Ave., Manila, Philippines, supported by a grant from the Research Grants Management Office of DLSU, grant no.: 02FR1TAY24-1TAY25, email: rafael.cantuba@dlsu.edu.ph}\\ ORCID: 0000-0002-4685-8761\orcidlink{0000-0002-4685-8761}}
\date{}

\begin{document}

\maketitle

\begin{abstract}
In measure and integration theory, the themes of extending premeasures to outer measures and extending Daniell linear functionals to Daniell integrals are governed by the foundational extension theorems of Carath\'{e}odory and Daniell. These are proven to be equivalent to Littlewood's Three Principles, which are known to be equivalent alternatives to a completeness axiom for an ordered field. The continuity theorems of measure and integration theory are consequently shown to be included in the list of equivalences.
\end{abstract}

\vspace{0.5em}
\noindent
\begin{quote}

\vspace{0.5em}
\end{quote}

\section{Introduction}

In the words of J.E. Littlewood, ``The extent of knowledge required (in measure and integration theory) is nothing like so great as is sometimes supposed'' \cite[p.~26]{lit44}, after which, J. E. Littlewood introduces his famous three principles. The quote in itself suggests that it is possible to have a fundamental heuristic that can aid one in understanding measure and integration theory, and in fact, as may be observed in most standard texts, Littlewood's Three Principles may indeed serve this purpose. In \cite{can24}, these three principles were proven to be equivalent to the Dedekind completeness of an ordered field, and this is an addition to the list of alternative completeness axioms that are obtained from classical real analysis theorems motivated by \cite{pro13,rie01,tei13}, with a survey done in \cite{dev14}. In this paper, we further show that the fundamental theorems that govern two classical extension processes in measure and integration theory are also valid completeness axioms for an ordered field. These two extension processes that we shall cover are of rather opposing natures. One is the very basic process of extending a measure, while the other is one famous approach of developing integration theory without using measures at all, or rather that the notion of measure comes after the notion of integral has been established. By the former, we mean the theory due to \Cara\!\!, which is regarding outer measures, while by the latter, we mean the theory of the Daniell integral. Both approaches have earned their status in measure and integration theory as fundamental frameworks. More specifically, we focus on theorems under these theories that state conditions when an outer measure may be obtained, or when particular stages of the Daniell extension process are possible.

Since the very goal of the \Cara and Daniell extension processes is to obtain an integral that ``commutes'' with function limits, the involvement of the so-called ``continuity theorems'' of measure and integration theory, which are the Monotone Convergence Theorem, Fatou's Lemma and the Dominated Convergence Theorem \cite[p.~179]{sum23}, is only natural. Furthermore, we show that when these theorems are stated as axioms for an ordered field, then each is also equivalent to Dedekind completeness. The reader may find in Section~\ref{PrelSec} how  preliminaries for measure and integration theory may be stated for an ordered field. The exact statements of our new completeness axioms are in Section~\ref{StateSec}, while the general scheme of implications may be found in Figure~\ref{TheFig}.

\section{Preliminaries}\label{PrelSec}

We shall make use of all the notation and definitions for analysis in an ordered field $\F$ as laid out in \cite[Section~2]{can24}. The preliminaries section in this paper shall hence focus on what has not been defined in \cite[Section~2]{can24} but shall be needed for our main result.

By a vector lattice on $\abClosed$, we shall mean a vector lattice of functions $\abClosed\into\F$ under the vector space operations of pointwise function addition and pointwise left multiplication by a constant from $\F$, and with the lattice operations as taking the pointwise maximum and minimum of two functions, denoted by $\vee$ and $\wedge$, respectively. A linear functional $I$ on a vector lattice $\stepLatt$ is \emph{positive} if the function inequality $f\geq 0$ implies $I(f)\geq 0$. In such a case, by linearity, $I$ has the ``monotonicity property:" $f\leq g$ implies $I(f)\leq I(g)$. A positive linear functional $I$ on $\stepLatt$ is a \emph{Daniell functional} on $\stepLatt$ if for any sequence $(\varphi_n)$ in $\stepLatt$ that monotonically decreases to the zero function, the sequence $(I(\varphi_n))$, which is a sequence in $\F$, monotonically decreases to $0\in\F$. Since $I$ has the aforementioned monotonicity property, we may simply replace $I(\varphi_n)\downarrow 0$ to $\nseqlim I(\varphi_n)=0$ in the previous definition of Daniell functional. The set $\istepLatt(I)$ of all functions $f:\abClosed\into\F$, for which there exists a monotonically increasing sequence $(\varphi_n)$ in $\stepLatt$ such that $(\varphi_n)$ monotonically increases to $f$, shall be called the \emph{first Daniell extension} of $\stepLatt$. 

A Daniell functional $I$ on $\stepLatt$ \emph{has a continuous extension to $\istepLatt(I)$} if for every $f\in\istepLatt$ and every sequence $(\varphi_n)$ in $\stepLatt$, if $\varphi_n\uparrow f$, then the sequence $(I(\varphi_n))$, which is a sequence in $\F$, monotonically increases to some element $I(f)$ of $\F$. Throughout, the set of all positive integers shall be denoted by $\N$.
\begin{proposition}\label{htophiProp} If, for each $f\in\istepLatt$, there exists a sequence $(h_n)$ in $\stepLatt$ that monotonically increases to $f$ for which the sequence $(I(h_n))$ converges to some $I(f)\in\F$, then $I$ has a continuous extension to $\istepLatt$. 
\end{proposition}
\begin{proof} If $(x_m)$ is a monotonically increasing sequence in $\stepLatt$ and $L\in\F$, then, by a routine argument, the sequence $(L-(x_m\wedge L))$ monotonically decreases to $0$. Since $I$ is a Daniell functional, $\displaystyle{\lim_{m\rightarrow\infty}}I(L-(x_m\wedge L))=0$, and by the linearity of $I$,
\begin{flalign}
    &&\displaystyle{\lim_{m\rightarrow\infty}}I(x_m\wedge L)&=I(L),&((x_m)\in\stepLatt^\N,\  \forall m\in\N\  x_m\leq x_{m+1}).\label{wedgeLim}
\end{flalign}

If $(\varphi_k)$ is a sequence in $\stepLatt$ such that $\varphi_k\uparrow f$, then setting $m=n$, $x_m=h_n$, $L=\varphi_k$ in \eqref{wedgeLim}, we obtain 
\begin{equation}
    \nseqlim I(h_n\wedge\varphi_k)=I(\varphi_k),
\end{equation}
so, given $\varepsilon>0$, there exists $N\in\N$ such that $I(f)-\frac{\varepsilon}{2}<I(h_N)$, which implies
\begin{equation}
    I(f)-\varepsilon<I(h_N)-\frac{\varepsilon}{2}.\label{DaniellExtend1}
\end{equation}
Setting $m=k$, $x_m=\varphi_k$, $L=h_N$ in \eqref{wedgeLim}, we find that there exists $K\in\N$ such that $k\geq K$ implies $I(h_N)-\frac{\varepsilon}{2}<I(h_N\wedge \varphi_k)$, which, in conjunction with \eqref{DaniellExtend1}, gives us
\begin{flalign}
    && I(f)-\varepsilon&<I(h_N\wedge \varphi_k), & (k\geq K).\label{DaniellExtend2}
\end{flalign}
By an epsilon argument, the fact that $\varphi_k\uparrow f$ implies $\varphi_k\leq f$ for all $k\in\N$, which further implies $h_N\wedge\varphi_k\leq \varphi_k\leq f$ for all $k\in\N$. By the linearity of $I$, we have $I(h_N\wedge\varphi_k)\leq I(\varphi_k)\leq I(f)$ for all $k\in\N$, which, in conjunction with \eqref{DaniellExtend2}, gives us $I(f)-\varepsilon<I(\varphi_k)\leq I(f)$ for all $k\geq K$. Therefore, the sequence $(I(\varphi_k))$ converges to $I(f)\in\F$, and $I$ has a continuous extension to $\istepLatt$.
\end{proof}

Even though the arguments in the above proof of Proposition~\ref{htophiProp} parallel standard constructions in measure and integration theory, we review them here explicitly to emphasize that none of the steps depend on any completeness axiom for the ordered field $\F$. To keep this methodological constraint in view, we highlight the following.

\begin{remark}\label{heuRem}
As a guiding principle, any proof of a measure-theoretic or integration-theoretic statement concerning an arbitrary ordered field $\F$ must be established independently of any completeness axiom for $\F$, unless the current goal is precisely to prove that the statement in question is equivalent to a completeness axiom.
\end{remark}

As a consequence of Proposition~\ref{htophiProp}, if $I$ has a continuous extension to $\istepLatt(I)$, then $I$ has indeed been extended to a function $\istepLatt\into\F$, which is $f\mapsto\nseqlim I(\varphi_n)$. We assume no confusion arises in denoting the continuous extension of $I$ to $\istepLatt$ also by $I$.

Given a vector lattice $\stepLatt$ on $\abClosed$ and a Daniell functional $I$ on $\stepLatt$ that has a continuous extension to $\istepLatt(I)$, the \emph{lower Daniell functional} $\underline{I}$ and \emph{upper Daniell functional} $\overline{I}$ of $I$ are defined, for an arbitrary function $f:\abClosed\into\F$, as
\begin{enumerate}
    \item $\underline{I}(f) := \nseqlim I(\phi_n)$, given a sequence $(\phi_n)$ in $\stepLatt$ such that $\phi_n \uparrow \phi \le f$;
    \item $\overline{I}(f) := -\nseqlim I(\psi_n)$, given a sequence $(\psi_n)$ in $\stepLatt$ such that $\psi_n \uparrow \psi \le -f$.
\end{enumerate}

The set of all positive elements of $\F$ shall be denoted by $\P$. A function $f:\abClosed\into\F$ is \emph{Daniell-integrable} with respect to $(\stepLatt,I)$ if for every $\varepsilon\in\P$, there exist $\phi,\psi\in \istepLatt(I)$ such that $-\phi \le f \le \psi$ and $I(\phi)+I(\psi)<\varepsilon$. The vector space (over $\F$) of all functions $\abClosed\into\F$ Daniell-integrable with respect to $(\stepLatt,I)$ is denoted by $L^1(\stepLatt,I)$, and shall henceforth be referred to as the \emph{second Daniell extension} of $\stepLatt$.

We say that $I$ is \emph{extendable to a Daniell integral} on $L^1(\stepLatt,I)$ if for any $f:\abClosed\into\F$, we have $\widetilde{I}(f):=\underline{I}(f) = \overline{I}(f)\in\F$ for all $f\in L^1(\stepLatt,I)$. In such a case, $\widetilde{I}(f)$ is called the \emph{Daniell integral} of $f$ over $\abClosed$.

From \cite[p.~247]{str20}, if $I$ a Daniell functional on a vector lattice $\stepLatt$ on $\abClosed$, then the triple $(\abClosed,\stepLatt,I)$ is called an \emph{integration theory}, and Daniell's Extension Theorem states that, for the case $\F=\R$, the triple $(\abClosed,L^1(\stepLatt,I),\widetilde{I})$ is also an integration theory in which $I$ and $\widetilde{I}$ agree on $\stepLatt\sub L^1(\stepLatt,I)$. The term integration theory seems grandiose, and so we shall state our results in terms of the simpler building blocks of the Daniell extension process. We now go back to the case when $\F$ is an arbitrary ordered field.

A collection $\abAlg$ of subsets of $\abClosed$ is a \emph{ring} on $\abClosed$ if, for every $A,B\in\abAlg$, the intersection $A\cap B$ and the symmetric difference $A\Delta B$ are also in $\abAlg$. A ring on $\abClosed$ that contains $\abClosed$ as one of its elements is an \emph{algebra} on $\abClosed$.  
\begin{proposition} The collection of all measurable functions $\abClosed\into\F$ in the sense of \cite[p.~114]{can24} and the collection of all step functions $\abClosed\into\F$ are vector lattices, and the collection of all measurable subsets of $\abClosed$ in the sense of \cite[p.~115]{can24} is an algebra on $\abClosed$.
\end{proposition}
\begin{proof} If $c\in\F$, and if $f$ and $g$ are measurable functions $\abClosed\into\F$, then there exist monotonically decreasing sequences $(\varphi_n)$ and $(\psi_n)$ of step functions $\abClosed\into\F$ that converge pointwise to $f$ and $g$, respectively, almost everywhere on $\abClosed$. Without loss of generality, we may assume the underlying set of measure zero to be the same set $Z$. That is, $x\in\abClosed\setdiff Z$ implies $\nseqlim\varphi_n(x)=f(x)$ and $\nseqlim\psi_n(x)=g(x)$. Since the rules for operations on convergent sequences are valid in $\F$, we have $\nseqlim[\varphi_n(x)+\psi_n(x)]=f(x)+ g(x)$, $\nseqlim\varphi_n(x)\vee\psi_n(x)=f(x)\vee g(x)$, $\nseqlim\varphi_n(x)\wedge\psi_n(x)=f(x)\wedge g(x)$, $\nseqlim |c|\varphi_n(x)=|c|f(x)$ and $\nseqlim c=c$, where the sequences $(\varphi_n+\psi_n)$, $(\varphi_n\vee\psi_n)$, $(\varphi_n\wedge\psi_n)$, $(|c|\varphi_n)$ and the constant sequence $(c)$ are all monotonically decreasing. Thus, $f+g$, $f\vee g$, $f\wedge g$, $|c|f$ and the constant function $c$ are all measurable. Since the difference of two nonnegative measurable functions is measurable \cite[p.~115]{can24}, so are $|c|(f\vee 0)-|c|(f\wedge 0)$ and $|c|(f\wedge 0)-|c|(f\vee 0)$, where $cf$ is always one of these two functions. Thus, the collection of all measurable functions $\abClosed\into\F$ form a vector lattice, a vector sublattice of which, is the collection of all step functions $\abClosed\into\F$. If $E$ and $F$ are measurable subsets of $\abClosed$, the characteristic functions $\chi_E$ and $\chi_F$ are measurable, and hence, so are $\chi_E\wedge\chi_F=\chi_{E\cap F}$, $(\chi_E-\chi_F)\vee(\chi_F-\chi_E)=|\chi_E-\chi_F|=\chi_{E\Delta F}$ and $1=\chi_{\abClosed}$ are all measurable. That is, $E\cap F$, $E\Delta F$ and $\abClosed$ are measurable. Therefore, the measurable subsets of $\abClosed$ form an algebra.
\end{proof}
\begin{proposition} The function $\varphi\mapsto\int_a^b\varphi$ is a positive linear functional on the vector lattice of all step functions $\abClosed\into\F$. 
\end{proposition}
\begin{proof} Denote two step functions $\abClosed\into\F$ by $\phi=\sum_{h=1}^mc_h\chi_{A_h}$ and $\psi=\sum_{k=1}^ne_k\chi_{B_k}$, where the sets $A_h$ are pairwise disjoint sets each of which is either a singleton or an open interval with $\abClosed=\bigcup_{h=1}^mA_h$, and similarly for the sets $B_k$. The sets
\begin{flalign}
    && A_h\setdiff\lpar\bigcup_{t=1}^nB_t\rpar,\quad B_k\setdiff\lpar\bigcup_{t=1}^mA_t\rpar, &&\nonumber\\
    && A_h\cap B_k, &&(h\in\{1,2,\ldots,m\},\  k\in\{1,2,\ldots,n\}),\nonumber
\end{flalign}
are pairwise disjoint, and we have the relations
\begin{flalign}
    && \chi_{A_h}&=\chi_{A_h\setdiff\lpar\bigcup_{t=1}^nB_t\rpar}+\sum_{k=1}^n\chi_{A_h\cap B_k}, &(h\in\{1,2,\ldots,m\}),\nonumber\\
    && \chi_{B_k}&=\chi_{B_k\setdiff\lpar\bigcup_{t=1}^mA_t\rpar}+\sum_{h=1}^m\chi_{A_h\cap B_k}, &(k\in\{1,2,\ldots,n\}),\nonumber
\end{flalign}
for the characteristic functions. Thus, we may represent $\phi$ and $\psi$ as
\begin{flalign}
    && \phi&=\sum_{h=1}^mc_h\chi_{A_h\setdiff\lpar\bigcup_{t=1}^nB_t\rpar}+\sum_{k=1}^n0\chi_{B_k\setdiff\lpar\bigcup_{t=1}^mA_t\rpar}+\sum_{h=1}^m\sum_{k=1}^nc_h\chi_{A_h\cap B_k}, &\label{steprep1}\\
    && \psi&=\sum_{h=1}^m0\chi_{A_h\setdiff\lpar\bigcup_{t=1}^nB_t\rpar}+\sum_{k=1}^ne_k\chi_{B_k\setdiff\lpar\bigcup_{t=1}^mA_t\rpar}+\sum_{k=1}^n\sum_{h=1}^me_k\chi_{A_h\cap B_k}, &\label{steprep2}
\end{flalign}
and using these representations of $\phi$ and $\psi$, the coefficients for the corresponding representations for $\phi+\psi$, $\kappa\phi$ (for all $\kappa\in\F$), $\phi\vee\psi$ and $\phi\wedge\psi$ may be obtained by performing these operations on the corresponding coefficients, term-by-term, according to how they are written in \eqref{steprep1}--\eqref{steprep2}. Also, the relations \eqref{steprep1}--\eqref{steprep2} may be used to compute for $\int_a^b(\phi+\psi)$ and $\int_a^b(\kappa\phi)$ and show that these are equal to $\int_a^b\phi+\int_a^b\psi$ and $\kappa \int_a^b\phi$, respectively. If $\phi\geq 0$, then, for each $H\in\{1,2,\ldots,m\}$, if $x_H\in A_H$, then $0\leq \phi(x_H)=\sum_{h=1}^mc_h\chi_{A_h}(x_h)=c_H\cdot 1$, because the sets $A_h$ are pairwise disjoint. Thus, $c_H\geq 0$ for all $H$, and because each $\chi_{A_H}$ is a nonnegative-valued function, $\int_a^b\phi$ is the finite sum of the nonnegative products $c_H\chi_{A_H}$, so $\int_a^b\phi\geq 0$. Therefore, $\varphi\mapsto\int_a^b\varphi$ is a positive linear functional.
\end{proof}
A \emph{premeasure} on an algebra $\abAlg$ on $\abClosed$ is a function $\mu:\abAlg\into\P\cup\{0,\infty\}$ that sends the empty set to zero, and satisfies \emph{$\sigma$-additivity}, which means that for any pairwise disjoint sequence $(E_n)$ in $\abAlg$, if $\bigcup_{n=1}^\infty E_n\in\abAlg$ and if the series $\sum_{n=1}^\infty\mu(E_n)$ converges in $\F$, then $\mu\lpar\bigcup_{n=1}^\infty E_n\rpar=\sum_{n=1}^\infty\mu(E_n)$. A premeasure $\mu$ on an algebra on $\abClosed$ is a \emph{finite} premeasure if $\mu(\abClosed)<\infty$.  An algebra that is closed under the formation of countable unions is a \emph{$\sigma$-algebra}, and a premeasure defined on a $\sigma$-algebra is a \emph{measure}.  Given a premeasure $\mu$ on an algebra $\abAlg$ on $\abClosed$, a subset $E$ of $\abClosed$ \emph{has $\mu^*$-measure} if there exists $\mu^*(E)\in\F$ such that
\[\mu^*(E)=\inf\left\{\sum_{n=1}^\infty\mu(E_n)  :\  (E_n)\in\abAlg^\N,\ E\sub\bigcup_{n=1}^\infty E_n,\  \sum_{n=1}^\infty\mu(E_n)\in\F\right\},\]
while $E\sub\abClosed$ is said to be \emph{$\mu^*$-measurable} if, for any $A\sub\abClosed$ that has $\mu^*$-measure, the intersection $A\cap E$ and the set difference $A\setdiff E$ also have $\mu^*$-measure, and furthermore,
\[\mu^*(A)\geq\mu^*(A\cap E)+\mu^*(A\setdiff E).\]
The function $\mu^*:E\mapsto\mu^*(E)$ is the \emph{\Cara outer measure} induced by the premeasure $\mu$.

\section{Statements of the Dedekind completeness axioms}\label{StateSec}

We are now ready to give the statements of the measure-theoretic and integration-theoretic axioms for $\F$ that we shall prove to be equivalent.

\begin{itemize}
\item[\CPP] \emph{Carath\'{e}odory Premeasure Property.} Given $\abClosed\sub\F$, the function\linebreak $E\mapsto\int_a^b\chi_E$ is a premeasure on the algebra of all measurable subsets of $\abClosed$.
    \item[\CEP]\emph{Carath\'{e}odory Extension Property.} Given $\abClosed\sub\F$, the function\linebreak $E\mapsto\int_a^b\chi_E$ is extendable to a \Cara outer measure $\mu^*$, and the collection of all $\mu^*$-measurable subsets of $\abClosed$ is a $\sigma$-algebra that contains the algebra of all measurable subsets of $\abClosed$.
    \item[\DFP]\emph{Daniell Functional Property.} Given $\abClosed\sub\F$,  the positive linear functional $\varphi\mapsto\int_a^b\varphi$ is a Daniell functional on the vector lattice of all step functions $\abClosed\into\F$.
    \item[\iDEP]\emph{First Daniell Extension Property.} Given $\abClosed\sub\F$,  the positive linear functional $I_{\abClosed}:\varphi\mapsto\int_a^b\varphi$ on the vector lattice $\stepLatt(\abClosed)$ of all step functions $\abClosed\into\F$ has a continuous extension to the first Daniell extension $\istepLatt(I_{\abClosed})$ of $\stepLatt(\abClosed)$. 
\item[\iiDEP]\emph{Full Daniell Extension Property.} Given $\abClosed\sub\F$, the positive linear functional $I_{\abClosed}:\varphi\mapsto\int_a^b\varphi$ on the vector lattice $\stepLatt(\abClosed)$ of all step functions $\abClosed\into\F$ is extendable to a Daniell integral on the second Daniell extension $L^1(\stepLatt(\abClosed),I_{\abClosed})$ of $\stepLatt(\abClosed)$. 

    \item[\MCP]\emph{Monotone Convergence Property.} Given $\abClosed\sub\F$, for each monotonically increasing sequence $(\varphi_n)$ of nonnegative Lebesgue measurable functions $\varphi_n:\abClosed\into\F$ that converges pointwise to $f:\abClosed\into\F$ almost everywhere on $\abClosed$, we have $\int_a^bf=\nseqlim\int_a^b\varphi_n$.
    
    \item[\FCP]\emph{Fatou Convergence Property.} Given $\abClosed\sub\F$, for each sequence $(\varphi_n)$ of nonnegative Lebesgue measurable functions $\varphi_n:\abClosed\into\F$ that converges pointwise to $f:\abClosed\into\F$ almost everywhere on $\abClosed$, the quantity $\nseqliminf\int_a^b\varphi_n$ exists in $\F$. Furthermore, $\int_a^bf\leq\nseqliminf\int_a^b\varphi_n$.
    \item[\DCP]\emph{Dominated Convergence Property.} Given $\abClosed\sub\F$, for each sequence $(\varphi_n)$ of Lebesgue measurable functions $\varphi_n:\abClosed\into\F$ that converges pointwise to $f:\abClosed\into\F$ almost everywhere on $\abClosed$, if there exists a Lebesgue measurable function $g:\abClosed\into\F$ such that $|\varphi_n|\leq g$ almost everywhere on $\abClosed$ for all $n\in\N$, then $\int_a^bf=\nseqlim\int_a^b\varphi_n$.

    \item[\LIP]\emph{Lebesgue Integral Property.} Given $\abClosed\sub\F$, every Lebesgue measurable function has a Lebesgue integral.

        \item[\LMP]\emph{Lebesgue Measure Property.} Given $\abClosed\sub\F$, every  measurable subset of $\abClosed$ has Lebesgue measure.

\end{itemize}

Each of the above statements has a general form \Star, which is the statement that results when we replace $\abClosed$ by an arbitrary subset $X$ of $\F$, or the function $E\mapsto\int_a^b\chi_E$ by an arbitrary measure $\mu$ on $X$, or $\elemI$ by an arbitrary positive linear functional $I$ on an arbitrary vector lattice on $X$, and so on. This general form \Star\  is a theorem for the case $\F=\R$, which is the unique Dedekind complete field up to isomorphism, so \Star\  follows from Dedekind completeness, which implies the specific case, which is, say, when $X=\abClosed$, $\mu:E\mapsto\int_a^b\chi_E$, $I=\elemLatt$, and so on. Thus, our goal, similar to that done in \cite[pp.~116--117]{can24}, is to show that the specific forms of the statements, the ones written above, imply Dedekind completeness of $\F$, and this completes the circle of equivalence. In particular, \LIP\  and \LMP, together with Littlewood's Three Principles, have already been proven to be equivalent to Dedekind completeness in \cite{can24}, so we have the more specific goal of proving that each of the other statements is a sufficient condition for \LIP\  or \LMP. We refer the reader to \cite{can24} for a full discussion of the involvement of Littlewood's Three Principles in reverse mathematics. 

\section{Logical relationships between the statements}

The general plan for proving the equivalence of the statements in Section~\ref{StateSec} with the Dedekind completeness of $\F$ is summarized in the following figure.

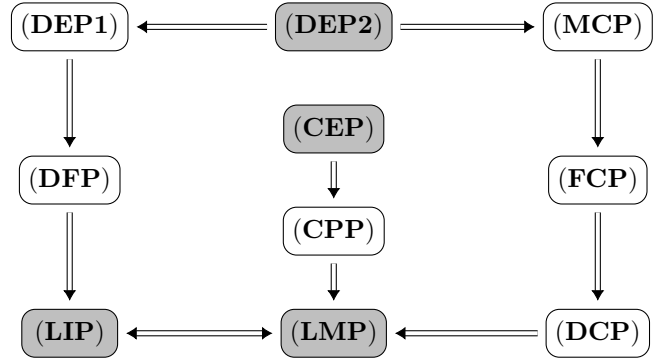
\begin{figure}[H]
\centering
\begin{tikzpicture}[
scale=0.9,             
    transform shape,       
    double arrow/.style={
        double, 
        double distance=1.5pt, 
        draw=#1, 
        ->, 
        >={Latex[length=4pt, width=5pt, color=#1]}, 
        shorten >= 3pt, 
        shorten <= 3pt,
        rounded corners=10pt 
    },
    double iffarrow/.style={
        double, 
        double distance=1.5pt, 
        draw=#1, 
        <->, 
        >={Latex[length=4pt, width=5pt, color=#1]}, 
        shorten >= 3pt, 
        shorten <= 3pt,
        rounded corners=10pt 
    },
    vertex/.style={
        draw, 
        rectangle, 
        rounded corners=5pt, 
        minimum width=40pt, 
        minimum height=20pt
    }
]

        \node[vertex, draw=black, fill=white] (iDEP) at (0,15) {\iDEP};
        \node[vertex, draw=black, fill=white] (DFP) at (0,25.5*0.5) {\DFP};
        \node[vertex, draw=black, fill=lightgray] (LIP) at (0,10.5) {\LIP};

        \node[vertex, draw=black, fill=lightgray] (iiDEP) at (3.875,15) {\iiDEP};
        \node[vertex, draw=black, fill=lightgray] (CEP) at (3.875,13.5) {\CEP};
        \node[vertex, draw=black, fill=white] (CPP) at (3.875,12) {\CPP};
        \node[vertex, draw=black, fill=lightgray] (LMP) at (3.875,10.5) {\LMP};

        \node[vertex, draw=black, fill=white] (MCP) at (2*3.875,15) {\MCP};
        \node[vertex, draw=black, fill=white] (FCP) at (2*3.875,25.5*0.5) {\FCP};
        \node[vertex, draw=black, fill=white] (DCP) at (2*3.875,10.5) {\DCP};

     \draw[double arrow=black] (iiDEP) -- (iDEP);
    \draw[double arrow=black] (iDEP) -- (DFP);
    \draw[double arrow=black] (DFP) -- (LIP);
    \draw[double iffarrow=black] (LIP) -- (LMP);

    \draw[double arrow=black] (iiDEP) -- (MCP);
    \draw[double arrow=black] (MCP) -- (FCP);
    \draw[double arrow=black] (FCP) -- (DCP);
    \draw[double arrow=black] (DCP) -- (LMP);

    \draw[double arrow=black] (CEP) -- (CPP);
    \draw[double arrow=black] (CPP) -- (LMP);
\end{tikzpicture}
\caption{Extension and convergence theorems in measure and integration theory, stated as axioms for an ordered field, are each a sufficient condition for the Lebesgue Integral Property and Lebesgue Measure Property, which were proven in \cite{can24} to be equivalent to Dedekind completeness.}\label{TheFig}
\end{figure}

In accordance with the explanation in the previous section, the diagram in Figure~\ref{TheFig} shows that our strategy is to show that each statement in question is a sufficient condition for \LIP\  and \LMP, which have been proven in \cite[Theorem~1]{can24} to be equivalent to Dedekind completeness. That is, if we succeed in proving every one-directional implication in Figure~\ref{TheFig}, then Dedekind completeness implies an arbitrary statement from Section~\ref{StateSec} which implies one of \LIP, \LMP, which implies Dedekind completeness. Hence, each such statement is an equivalent completeness axiom for $\F$. Since some chains of implications can be formed between statements, we highlighted in gray the strongest statements, which are \iiDEP\  and \CEP, and also the weakest statements, which should be \LIP\  and \LMP, as previously explained. Some of the implications are immediate from the preliminaries as we shall now show.

\section{Implications immediate from the theory}\label{ImpliesSec}

If a positive linear functional $\elemI:\varphi\mapsto\int_a^b\varphi$ on the vector lattice $\elemLatt$ is extendable to a Daniell integral on the second Daniell extension of $\elemLatt$, then the first step of this extension process must be possible, so $\elemI$ has a continuous extension to the first Daniell extension of $\elemLatt$. Necessarily, $\elemI$ in this case must be a Daniell functional, so we have
\[\iiDEP\implies\iDEP\implies\DFP.\]
If the function $E\mapsto\int_a^b\chi_E$ is extendable to a \Cara outer measure on $\abClosed$, then this extension process necessitates that $E\mapsto\int_a^b\chi_E$ is a premeasure, and if it is so, then every element of the algebra of all measurable subsets $E$ of $\abClosed$ must be assigned an integral $\int_a^b\chi_E\in\F$. Thus,
\[\CEP\implies\CPP\implies\LMP.\]
The very essence of the Daniell extension framework for integration theory is that if the extension of a Daniell functional exists for the second Daniell extension of the underlying vector lattice, then the Daniell integral satisfies a Monotone Convergence Principle, which implies an analog of Fatou's lemma, which in turn implies an analog of the Dominated Convergence Principle. See, for instance, \cite[Propositions~10,12,13]{roy88}. The statements of these propositions may be instantiated on the specific integrals and functions in our statements of \iiDEP, \MCP, \FCP\  and \DCP. Thus, the implications
\[\iiDEP\implies\MCP\implies\FCP\implies\DCP,\]
are immediate. As mentioned earlier, the equivalence of \LIP\  and \LMP\  is part of the subject of the paper \cite{can24}. Thus, by inspection of Figure~\ref{TheFig}, only two implications remain, which are ``\DFP\  implies \LIP,'' and ``\DCP\  implies \LMP.''

\section{Summary and final remarks}

To complete the proofs of the implications in Figure~\ref{TheFig}, we state and prove our main result.

\begin{theorem} The statements in Figure~\ref{TheFig} are each equivalent to the Dedekind completeness of the ordered field $\F$.
\end{theorem}
\begin{proof} As discussed in Section~\ref{ImpliesSec}, only two implications remain, which we now prove. Consider an interval $\abClosed\sub\F$.\\

\noindent$\DFP\implies\LIP$. Let $f:\abClosed\into\F$ be a Lebesgue measurable function, and suppose $(\varphi_n)$ and $(\psi_n)$ are a monotonically decreasing sequences of step functions that converge to $f$ almost everywhere on $\abClosed$. The sequence $(\varphi_n\wedge\psi_n)$ also monotonically decreases to $f$, and the sequence $(\varphi_n-\varphi_n\wedge\psi_n)$ monotonically decreases to the zero function. If \DFP\ is true, then we have $\nseqlim\int_a^b(\varphi_n-\varphi_n\wedge\psi_n)=0$, which implies $\nseqlim\int_a^b\varphi_n=\nseqlim\int_a^b(\varphi_n\wedge\psi_n)$. Since $(\varphi_n)$ and $(\psi_n)$ are arbitrary, we may interchange $\varphi_n$ and $\psi_n$ in the previous equation, and we further have
\[\nseqlim\int_a^b\varphi_n=\nseqlim\int_a^b(\varphi_n\wedge\psi_n)=\nseqlim\int_a^b\psi_n.\]
That is, for any two sequences of step functions that monotonically decrease to $f$ almost everywhere on $\abClosed$, the corresponding sequences of integrals converge to the same value. From the definition in \cite[pp.~114--115]{can24}, this means that $f$ has a Lebesgue integral over $\abClosed$, which proves \LIP.\\

\noindent$\DCP\implies\LMP$. Let $E$ be a measurable subset of $\abClosed$, and let $(\varphi_n)$ be a sequence of step functions (which are measurable functions) $\abClosed\into\F$ that monotonically decrease to $\chi_E$ almost everywhere on $\abClosed$. Such a sequence is guaranteed to exist by the definition of measurable subset of $\abClosed$ from \cite[pp.~114--115]{can24}. The fact that $(\varphi_n)$ is monotonically decreasing may be used in a routine epsilon argument, valid in the ordered field $\F$, to show that $\chi_E\leq\varphi_n$ for all $n\in\N$. Also for the same reason that $(\varphi_n)$ is monotonically decreasing, we further have $\chi_E\leq \varphi_n\leq \varphi_1$, but since the left-most function in these inequalities is nonnegative, we may replace $\varphi_n$ by $|\varphi_n|$. That is, $|\varphi_n|\leq \varphi_1$ for all $n\in\N$. If \DCP\  holds, then $\int_a^b\chi_E=\nseqlim\int_a^b\varphi_n$. Hence, $\chi_E$ has a Lebesgue integral over $\abClosed$, or equivalently, $E$ has Lebesgue measure, which proves \LMP.
\end{proof}

We hope that by proving the equivalence of the statements in Figure~\ref{TheFig}, we have contributed a significant heuristic to contemporary integration theory that clarifies exactly at which parts of the machinery of measure and integration the completeness of the real field is needed, and up to what point proofs can be handled by only the properties of the underlying constructs, such as Daniell functionals or premeasures, and we conjecture that this perspective can further elucidate more abstract aspects of measure and integration theory.


\begin{thebibliography}{9}

\bibitem{can24}
R. R. S. Cantuba, \textit{Littlewood's principles in reverse real analysis}, Real Anal. Exchange, \textbf{49(1)} (2024), 111--122.

\bibitem{dev14}
M. Deveau and H. Teismann, \textit{$72+42$: characterizations of the completeness and Archimedean properties of ordered fields}, Real Anal. Exchange, \textbf{39(2)} (2013/14), 261--304.


\bibitem{lit44}
J. E. Littlewood, \textit{Lectures on the Theory of Functions}, Oxford University Press, Oxford, 1944.

\bibitem{pro13}
J. Propp, \textit{Real analysis in reverse}, Amer. Math. Monthly, \textbf{120(5)} (2013), 392--408.

\bibitem{rie01}
O. Riemenschneider, \textit{37 elementare axiomatische Charakterisierungen des reellen Zahlk\"orpers}, Mitt. Math. Ges. Hamburg, \textbf{20} (2001), 71--95. (English translation available at \url{https://www.math.uni-hamburg.de/home/riemenschneider/axioms.pdf})

\bibitem{roy88}
H. L. Royden, \textit{Real Analysis}, Third edition. Macmillan Publishing Company, New York, 1988.

\bibitem{str20}
D. W. Stroock, \textit{Essentials of Integration Theory for Analysis}, Second edition. Grad. Texts Math., Vol. 262. Springer, Cham, 2020.

\bibitem{sum23}
S. Majumdar, \textit{Extension of the continuity theorems of Lebesgue integration}, Real Anal. Exchange, \textbf{48(1)} (2023), 179--200.

\bibitem{tei13}
H. Teismann, \textit{Toward a more complete list of completeness axioms}, Amer. Math. Monthly, \textbf{120(2)} (2013), 99--115.

\end{thebibliography}
\end{document}